\documentclass[12pt,a4paper]{article}
\usepackage{pifont}
\usepackage{amsfonts}
\usepackage{mathrsfs}
\usepackage{amssymb}
\usepackage{amsmath}
\usepackage{cases}

\newtheorem{theorem}{Theorem}[section]
\newtheorem{definition}[theorem]{Definition}
\newtheorem{lemma}[theorem]{Lemma}
\newtheorem{corollary}[theorem]{Corollary}

\newenvironment{proof}{{\bf Proof.  }}{$\square$}
\makeatletter

\newcommand{\Rmnum}[1]{\expandafter\@slowromancap\romannumeral #1@}
\makeatother

\begin{document}
\title{Gr\"{o}bner-Shirshov Bases and Hilbert Series of Free Dendriform Algebras\footnote{Supported by the
NNSF of China  (No.10771077) and the NSF of Guangdong Province
 (No.06025062).}}
\author{
%L. A. Bokut \footnote {Supported by RFBR 01-09-00157, LSS{344.2008.1
%and SB RAS Integration grant No. 2009.97 (Russia).}}\\
%{\small \ School of Mathematical Sciences, South China Normal University}\\
%{\small Guangzhou 510631, P. R. China}\\
%{\small Sobolev Institute of Mathematics, Russian Academy of Sciences}\\
%{\small Siberian Branch, Novosibirsk 630090, Russia}\\
%{\small Email: bokut@math.nsc.ru}\\
%\\
 Yuqun Chen and Bin Wang\\
{\small \ School of Mathematical Sciences, South China Normal University}\\
{\small Guangzhou 510631, P. R. China}\\
{\small  yqchen@scnu.edu.cn}\\
{\small windywang2004@126.com}}

\date{}

\maketitle \noindent\textbf{Abstract:}    In this paper, we give a
Gr\"{o}bner-Shirshov basis of the free dendriform algebra as a
quotient algebra of an $L$-algebra. As applications, we obtain a
normal form of the free dendriform algebra. Moreover,  Hilbert
series and Gelfand-Kirillov dimension of  finitely generated free
dendriform algebras are obtained.


 \section{Introduction}

The theories of Gr\"obner-Shirshov bases and Gr\"obner bases were
invented independently by A.I. Shirshov (\cite{Sh}, 1962) for
(commutative, anti-commutative) non-associative alegbras, by H.
Hironaka (\cite{H}, 1964) for infinite series algebras (both formal
and convergent) and by B. Buchberger (first publication in
\cite{bu65}, 1965) for polynomial algebras. Gr\"obner--Shirshov
technique is very useful in the study of presentations of many kinds
of algebras defined by generators and defining relations.

An $L$-algebra (see \cite{Ph}) is a vector space over a field $k$
with two operations $\prec,~\succ$ satisfying one identity: $(x\succ
y)\prec z = x\succ (y\prec z)$. A dendriform algebra (see
\cite{Lo01,Lo02})is an $L$-algebra with two identities:  $(x \prec
y)\prec z = x \prec (y \prec z) + x \prec (y \succ z)$ and $x
\succ(y \succ z) = (x \succ y)\succ z + (x \prec y)\succ z.$

The Composition-Diamond lemma for  $L$-algebras is established in a
recent paper \cite{Huang}. In this paper, by using the
Composition-Diamond lemma for $L$-algebras in \cite{Huang}, we give
a Gr\"{o}bner-Shirshov basis of the free dendriform algebra as a
quotient algebra of a free $L$-algebra and then a normal form of a
free dendriform algebra is obtained. As applications, we obtain the
Hilbert series and Gelfand-Kirillov dimension of the free dendriform
algebra generated by a finite set.

\section{$L$-algebras}

We first introduce some concepts and results from the literature
which are related to the Gr\"{o}bner-Shirshov bases for
$L$-algebras. We will use some definitions and notations which are
mentioned in \cite{Huang}.

Let $k$ be a field, $X$  a set of variables,  $\Omega$  a set of
multilinear operations, and
$$
\Omega=\cup_{n\geq1}\Omega_n,
$$
where $\Omega_n=\{\delta_{i}^{(n)}|i\in I_n\}$ is the set of $n$-ary
operations, $n=1,2,\dots$. Now, we define ``$\Omega$-words".

Define
$$
(X,\Omega)_0 = X.
$$

For $m \geqslant 1$, define
$$
(X,\Omega)_m = X \cup \Omega((X,\Omega)_{m-1})
$$
where
$$
\Omega((X,\Omega)_{m-1}) = \cup_{t=1}^\infty \{
\delta_{i}^{(t)}(u_1, u_2, \dots, u_t) |~\delta_{i}^{(t)} \in
\Omega_t, u_j \in\ (X,\Omega)_{m-1}\}.
$$

Let
$$
(X,\Omega)=\bigcup_{m=0}^\infty (X,\Omega)_m.
$$
Then each element in $(X,\Omega)$ is called an $\Omega$-word.

\begin{definition}\cite{Ph}\label{lalg}
An $L$-algebra is a $k$-vector space $L$ equipped with two bilinear
operations $\prec,~\succ : L^{\bigotimes 2} \rightarrow L$ verifying
the so-called entanglement relation:
$$
(x \succ y)\prec z = x \succ ( y \prec z), \forall~x, y ,z \in L.
$$
\end{definition}

Let $\Omega=\{\prec, \succ\}$. In this case, we call an
$\Omega$-word as an $L$-word.
\begin{definition}\cite{Huang}\label{l}
An $L$-word $u$ is a normal $L$-word if $u$ is one of the following:
\begin{enumerate}
\item[i)]\  $u=x$, \ where $x\in X$.
\item[ii)]\ $u=v\succ w$, where $v$ and $w$ are normal $L$-words.
\item[iii)]\ $u=v\prec w$ with $v\neq v_1\succ v_2$, where $v_1,\ v_2,\ v,\ w$ are normal $L$-words.
\end{enumerate}
We denote $u$ by $[u]$ if $u$ is a normal $L$-word.
\end{definition}

We denote the set of all the normal $L$-words by $N$. Then, the free
L-algebra has an expression $L(X)=kN=\{\sum\alpha_iu_i~|~\alpha_i\in
k,\ u_i\in N\}$ with $k$-basis $N$ and the operations $\prec,\
\succ$: for any $u,v\in N$,
$$
u\prec v=[u\prec v], \ \ \ \ u\succ v=[u\succ v].
$$
Clearly, $[u\succ v]=u\succ v$ and
$$
[u\prec v]=\left\{
\begin{array} {ll} u\prec v &\mbox{if } u=u_1\prec u_2,\mbox{or } u\in X,\\
u_1\succ[u_2\prec v] &\mbox{if } u=u_1\succ u_2.
\end{array} \right.
$$

 Now, we order $N$ in the same way as in \cite{Huang}.

Let $X$ be a well ordered set. We denote $\succ$ by $\delta_1$,
$\prec$ by $\delta_2$. For any normal $L$-word $u$, define

\begin{numcases}{wt(u)=}
   (1,x),  & if $u=x \in X$;\nonumber\\
   (|u|, \delta_i, u_1, u_2), & if $u=\delta_i(u_1,u_2)\in
   N$,\nonumber
  \end{numcases}
where $|u|$ is the number of $x\in X$ in $u$. Then we order $N$ as
follows:
$$
u>v\Longleftrightarrow wt(u)>wt(v) ~~~~~~~lexicographically
$$
by induction on $|u|+|v|$, where $\delta_2>\delta_1$.

Let $\star \not\in X$. By a $\star$-$L$-word we mean any expression
in $( X \cup \{\star\},\{\prec,\succ\})$ with only one occurrence of
$\star$.

Let $u$ be a $\star$-$L$-word and $s \in L(X)$. Then we call $u|_s =
u|_{\star\mapsto s}$ an $s$-word in $L(X)$.

An $s$-word $u|_s$ is called a normal $s$-word if
$u|_{\overline{s}}\in N$.

It is shown in \cite{Huang} that the above ordering  on $N$ is
monomial in the sense that for any $\star$-$L$-word $w$ and any $u,\
v\in N$, $u>v$ implies $[w|_u]>[w|_v]$.

Assume that $L(X)$ is equipped with the monomial ordering $>$ as
above. For any $L$-polynomial $ f\in L(X)$, let $\bar{f}$ be the
leading normal $L$-word of $f$. If the coefficient of $\bar{f}$ is
$1$, then $f$ is called monic.

\begin{definition}\cite{Huang}
Let $f,~g \in L(X)$ are two monic polynomials.
\begin{enumerate}
\item[1)] \ Composition of right multiplication.

If $\overline{f}=u_1\succ u_2$ for some $u_1,\ u_2\in N$, then for
any $v\in N$, $f\prec v$ is called a composition of right
multiplication.

\item[2)] \  Composition of inclusion.

If $ w=\overline{f}=u|_{\overline{g}}$ where $u|_{g}$ is a normal
$g$-word, then
$$
(f,g)_{w}=f-u|_g
$$
is called the composition of inclusion and $w$ is called the
 ambiguity of the composition $(f,g)_\omega$.

\end{enumerate}

\end{definition}

\begin{definition}\cite{Huang}
Let the ordering on $N$ be as before, $S\subset L(X)$ a monic set
and $f,~g\in S$.

\begin{enumerate}
\item[1)]\ The composition of right multiplication $f\prec v$ is called
trivial modulo $S$, denoted by $f\prec v\equiv 0 \ mod(S)$, if
$$
f\prec v=\sum{\alpha}_iu_i|_{s_i},
$$
where each $\alpha_i\in k, \ s_i\in S, \ u_i|_{s_i}$ normal
$s_i$-word, and $u_i|_{\overline{s_i}}\leqslant \overline{f\prec
v}$.

\item[2)]\ The composition of inclusion $(f,g)_{w}$ is called trivial modulo
$(S, w)$, denoted by $(f,g)_{w}\equiv 0 \ mod(S,w)$,  if
$$
(f,g)_{w}=\sum{\alpha}_iu_i|_{s_i},
$$
where each $\alpha_i\in k, \ s_i\in S, \ u_i|_{s_i}$ normal
$s_i$-word, and $u_i|_{\overline{s_i}}<w$.

\end{enumerate}

$S$ is called a Gr\"{o}bner-Shirshov basis in $L(X)$ if any
composition of polynomials in $S$ is trivial modulo $S$ (and $w)$.
\end{definition}

\begin{theorem}\cite{Huang} \label{cdl}
(Composition-Diamond lemma for $L$-algebras) \ Let $S\subset L(X)$
be a monic set and the ordering on $N$ as before. Let $Id(S)$ be the
ideal of $L(X)$ generated by $S$. Then the following statements are
equivalent:
\begin{enumerate}
\item[(\Rmnum{1})] \ $S$ is a Gr\"{o}bner-Shirshov basis in $L(X)$.
\item[(\Rmnum{2})] \
$f\in Id(S)\Rightarrow \overline{f}=u|_{\overline{s}}$ for some
$s\in S$, where $u|_s$ is a normal $s$-word.
\item[(\Rmnum{3})] \
The set $ Irr(S)=\{u\in N|~u\neq v|_{\overline{s}},\ s\in S,\
v|_{s}\mbox{ is a normal } s\mbox{-word}\} $ is a $k$-basis of the
L-algebra $L(X|~S)=L(X)/Id(S)$.
\end{enumerate}
\end{theorem}

\section{Gr\"{o}bner-Shirshov bases for free dendriform algebras}

In this section, we give a Gr\"{o}bner-Shirshov basis of the free
dendriform algebra $DD(X)$ generated by $X$. As an application, we
obtain a normal form of $DD(X)$.

\begin{definition}\cite{Lo01}\label{d}
A dendriform algebra is a $k$-vector space $DD$ with two bilinear
operations $\prec,~\succ$ subject to the  three axioms below: for
any $x,~y,~z\in DD,$
\begin{enumerate}
\item[1)] $(x \succ y)\prec z = x \succ ( y \prec z),$
\item[2)] $(x \prec y)\prec z = x \prec (y \prec z) + x \prec (y \succ
z),$
\item[3)] $x \succ(y \succ z) = (x \succ y)\succ z + (x \prec y)\succ
z.$
\end{enumerate}
\end{definition}

Thus, any dendriform algebra is an $L$-algebra.

It is clear that the free dendriform algebra generated by $X$,
denoted by $DD(X)$, has an expression
\begin{eqnarray*}
L(X&|&(x \prec y)\prec z = x \prec (y \prec z) + x \prec (y \succ
z),\\
&&  x \succ(y \succ z) = (x \succ y)\succ z + (x \prec y)\succ z,\
x,y,z\in N).
\end{eqnarray*}

The following theorem gives a Gr\"{o}bner-Shirshov basis for
$DD(X)$.

\begin{theorem}\label{GSBl}
Let the ordering on $N$ be as before. Let
\begin{eqnarray*}
f_1(x,y,z)& =&(x\prec y)\prec z - x \prec (y \prec z) -x \prec (y
\succ z),\\
f_2(x,y,z)&=&(x\prec y)\succ z + (x \succ y)\succ z-x\succ (y\succ
z),\\
f_3(x,y,z,v)&=&((x\succ y)\succ z)\succ v - (x\succ y)\succ (z\succ
v) + (x\succ(y\prec z))\succ v.
\end{eqnarray*}
Then, $S=\{f_1(x,y,z),~f_2(x,y,z),~f_3(x,y,z,v)|~x,y,z,v \in N\}$ is
a Gr\"{o}bner-Shirshov basis in $L(X)$.
\end{theorem}

\begin{proof}
All the possible compositions of right multiplication in $S$ are as
follows.

1) $f_2(x,y,z)\prec u,\ u\in N$. We have
\begin{eqnarray*}
~~f_2(x,y,z) \prec u &=& ((x \prec y)\succ z) \prec u + ((x \succ
y) \succ z) \prec u - (x \succ(y \succ z)) \prec u \\
&=& (x \prec y)\succ(z \prec u) + (x \succ y)\succ (z \prec u) - x
\succ ((y \succ z)\prec u) \\
&\equiv & -(x \succ y)\succ(z \prec u) + x \succ(y \succ(z \prec u))
+ (x \succ y)\succ (z \prec u) \\
& & - x \succ((y \succ z)\prec u)\\
&\equiv& 0~mod(S).
\end{eqnarray*}

~2) $f_3(x,y,z,v)\prec u,\ u\in N$. We have
\begin{eqnarray*}
~~f_3(x,y,z,v)\prec u &=& (((x\succ y)\succ z)\succ v)\prec u -
((x\succ y)\succ (z\succ v))\prec u \\
& & + ((x\succ (y\prec z))\succ v)\prec u\\
&=& ((x\succ y)\succ z)\succ (v\prec u) - (x\succ y)\succ (z\succ
(v\prec u))\\
& & + (x\succ (y\prec z))\succ (v\prec u)\\
&\equiv & (x\succ y)\succ (z\succ (v\prec u)) - (x\succ (y\prec
z))\succ (v\prec u)\\
& &- (x\succ y)\succ (z\succ (v\prec u)) + ((x\succ (y\prec z))\succ
v)\prec u \\
&\equiv& 0~mod(S).
\end{eqnarray*}
We denote by $f_i\wedge f_j$ an inclusion composition of the
polynomials $f_i$ and $f_j, \ i,j=1,2,3$. All the possible
ambiguities in $S$ are listed as follows:

~3)  $f_1(x,y,z)\wedge f_1(a,b,c):$

~3.1\ $w_{3.1} = (x|_{((a \prec b)\prec c)}\prec y)\prec
z,$~~~~~~~~~~~~~~~~~~~~~~3.2\ $w_{3.2} = (x \prec y|_{((a \prec
b)\prec c)})\prec z,$

~3.3\ $w_{3.3} = (x \prec y)\prec z|_{((a \prec b)\prec
c)},$~~~~~~~~~~~~~~~~~~~~~~3.4\ $w_{3.4} = ((a\prec b)\prec c)\prec
z.$
\\

~4) $f_1(x,y,z)\wedge f_2(a,b,c):$

~4.1\ $w_{4.1} = (x|_{((a \prec b)\succ c)} \prec y)\prec
z$,~~~~~~~~~~~~~~~~~~~~~~4.2\ $w_{4.2} = (x \prec y|_{((a \prec
b)\succ c)})\prec z$,

~4.3\ $w_{4.3}= (x \prec y)\prec z|_{((a \prec b)\succ c)}$.
\\

~5)  $f_1(x,y,z)\wedge f_3(a,b,c,d):$

~5.1\ $w_{5.1} = (x|_{(((a\succ b)\succ c)\succ d)}\prec y)\prec
z$,~~~~~~~~~~~~~~~~~5.2\ $w_{3.2} = (x \prec y|_{(((a\succ b)\succ
c)\succ d)})\prec z$,

~5.3\ $w_{5.3} = (x\prec y)\prec z|_{(((a\succ b)\succ c)\succ d)}$.
\\

~6)  $f_2(a,b,c)\wedge f_1(x,y,z)$:

~6.1\ $w_{6.1} = (a|_{((x\prec y)\prec z)}\prec b)\succ
c,$~~~~~~~~~~~~~~~~~~~~~~6.2\ $w_{6.2} = (a\prec b|_{((x\prec
y)\prec z)})\succ c,$

~6.3\ $w_{6.3} = (a\prec b)\succ c|_{((x\prec y)\prec
z)},$~~~~~~~~~~~~~~~~~~~~~~6.4\ $w_{6.4} = ((x\prec y)\prec z)\succ
c.$
\\

~7) $f_2(a,b,c)\wedge f_2(x,y,z)$:

~7.1\ $w_{7.1} = (a|_{((x\prec y)\succ z)}\prec b)\succ
c$,~~~~~~~~~~~~~~~~~~~~~~~7.2\ $w_{7.2} = (a\prec b|_{((x\prec
y)\succ z)})\succ c$,

~7.3\ $w_{7.3} = (a\prec b)\succ c|_{((x\prec y)\succ z)}$.
\\

~8) $f_2(a,b,c)\wedge f_3(x,y,z,v)$:

~8.1\ $w_{8.1} = (a|_{(((x\succ y)\succ z)\succ v)}\prec b)\succ
c$,~~~~~~~~~~~~~~~~~8.2\ $w_{8.2} = (a\prec b|_{(((x\succ y)\succ
z)\succ v)})\succ c$,

~8.3\ $w_{8.1} = (a\prec b)\succ c|_{(((x\succ y)\succ z)\succ v)}$.
\\

 ~9) $f_3(x,y,z,v)\wedge f_1(a,b,c)$:

~9.1\ $w_{9.1} = ((x|_{((a \prec b)\prec c)}\succ y)\succ z)\succ
v$,~~~~~~~~~~~~~~~~~~~9.2\ $w_{9.2} = ((x\succ y|_{((a \prec b)\prec
c)})\succ z)\succ v$,

~9.3\ $w_{9.3} = ((x\succ y)\succ z|_{((a \prec b)\prec c)})\succ
v$,~~~~~~~~~~~~~~~~~~~9.4\ $w_{9.4} = ((x\succ y)\succ z)\succ
v|_{((a \prec b)\prec c)}$.
\\

~10) $f_3(x,y,z,v)\wedge f_2(a,b,c)$:

~10.1\ $w_{10.1} = ((x|_{((a \prec b)\succ c)}\succ y)\succ z)\succ
v$,~~~~~~~~~~~~~10.2\ $w_{10.2} = ((x\succ y|_{((a \prec b)\succ
c)})\succ z)\succ v$,

~10.3\ $w_{10.3} = ((x\succ y)\succ z|_{((a \prec b)\succ c)})\succ
v$,~~~~~~~~~~~~~10.4\ $w_{10.4} = ((x\succ y)\succ z)\succ v|_{((a
\prec b)\succ c)}$,

~10.5\ $w_{10.5} = (((a\prec b)\succ c)\succ z)\succ v$.
\\

~11) $f_3(x,y,z,v)\wedge f_3(a,b,c,d)$:

~11.1\ $w_{11.1} = ((x|_{(((a\succ b)\succ c)\succ d)}\succ y)\succ
z)\succ v$,

~11.2\ $w_{11.2} = ((x\succ y|_{(((a\succ b)\succ c)\succ
d)})\succ z)\succ v$,

~11.3\ $w_{11.3} = ((x\succ y)\succ z|_{(((a\succ b)\succ c)\succ
d)})\succ v$,

~11.4\ $w_{11.4} = ((x\succ y)\succ z)\succ v|_{(((a\succ b)\succ
c)\succ d)}$,

~11.5\ $w_{11.5} = (((a\succ b)\succ c)\succ d)\succ v$,

~11.6\ $w_{11.6} =((((a\succ b)\succ c)\succ d)\succ z)\succ v$.
\\

We will prove that all compositions are trivial mod$(S,w)$. Here,
for example, we only check Case 6.4, Case 10.5 and Case 11.6. The
others are easy to check.

%Case 3.4:
%\begin{eqnarray*}
%&&(f_1(x,y,z), f_1(a,b,c))_{w_{3.4}}\\
% &\equiv & (a\prec
%(b\prec c))\prec z + (a\prec (b\succ c))\prec z - (a\prec b)\prec (c\prec z)\\
%&&  - (a\prec b)\prec (c\succ z)\\
%&\equiv & a\prec ((b\prec c)\prec z) + a\prec ((b\prec c)\succ z) + a \prec (b\succ (c\prec z)) \\
%&& + a\prec ((b\succ c)\succ z)- a\prec (b\prec (c\prec z)) - a\prec (b\succ (c\prec z))\\
%&& - a\prec (b\prec (c\succ z)) - a\prec (b\succ (c\succ z))\\
%&\equiv & a\prec ((b\prec c)\prec z) + a\prec (b\succ (c\succ z)) +
%a\prec (b\succ (c\prec z))\\
%&& - a\prec (b\prec (c\prec z)) - a\prec (b\succ (c\prec z)) -
%a\prec (b\prec (c\succ z))\\
%&&- a\prec (b\succ (c\succ z))\\
%&\equiv & 0~mod(S, w_{3.4}).
%\end{eqnarray*}

Case 6.4:
\begin{eqnarray*}
&&(f_2(a,b,c), f_1(x,y,z))_{w_{6.4}}\\
&\equiv & (x\prec (y\prec z))\succ c + (x\prec (y\succ z))\succ c -
(x\prec y)\succ
(z\succ c)\\
&& + ((x\prec y)\succ z)\succ c\\
&\equiv & x\succ ((y\prec z)\succ c) - (x\succ (y\prec z))\succ c +
x\succ ((y\succ z)\succ c)\\
&& - (x\succ (y\succ z))\succ c - x\succ (y\succ(z\succ c)) +
(x\succ y)\succ (z\succ c)\\
&& + (x\succ (y\succ z))\succ c - ((x\succ y)\succ z)\succ c \\
&\equiv& x\succ (y\succ (z\succ c)) - (x\succ (y\prec z))\succ c -
x\succ
(y\succ (z\succ c))\\
&& + (x\succ y)\succ (z\succ c) - ((x\succ y)\succ z)\succ c\\
&\equiv & 0 ~mod(S, w_{6.4}).
\end{eqnarray*}

Case 10.5:
\begin{eqnarray*}
&&(f_3(x,y,z,v), f_2(a,b,c))_{w_{10.5}}\\
&\equiv & ((a\succ
(b\succ c))\succ z)\succ v - (((a\succ b)\succ c)\succ z)\succ v \\
&& - ((a\prec b)\succ c)\succ (z\succ v) + ((a\prec b)\succ (c\prec z))\succ v\\
&\equiv & (a\succ (b\succ c))\succ (z\succ v) - (a\succ ((b\succ
c)\prec z))\succ v \\
&& - ((a\succ b)\succ c)\succ (z\succ v) + ((a\succ b)\succ (c\prec
z))\succ v\\
&& - (a\succ (b\succ c))\succ
(z\succ v) + ((a\succ b)\succ c)\succ (z\succ v)\\
&& + (a\succ (b\succ (c\prec z))\succ v - ((a\succ b)\succ (c\prec
z))\succ v\\
&\equiv & 0~mod(S, w_{10.5}).
\end{eqnarray*}

Case 11.6:
\begin{eqnarray*}
&&(f_3(x,y,z,v), f_3(a,b,c,d))_{w_{11.6}}\\
&\equiv & (((a\succ
b)\succ (c\succ d))\succ z)\succ v - (((a\succ (b\prec c))\succ
d)\succ z)\succ v\\
&& - (((a\succ b)\succ c)\succ d)\succ (z\succ v) + (((a\succ
b)\succ c)\succ (d\prec z))\succ v\\
&\equiv & ((a\succ b)\succ (c\succ d))\succ (z\succ v) - ((a\succ
b)\succ ((c\succ d)\prec z))\succ v\\
&& - ((a\succ (b\prec c))\succ d)\succ (z\succ v) + ((a\succ (b\prec
c))\succ (d\prec z))\succ v\\
&& - ((a\succ b)\succ (c\succ d))\succ (z\succ v) + ((a\succ (b\prec
c))\succ d)\succ (z\succ v)\\
&& + ((a\succ b)\succ (c\succ (d\prec z))\succ v - ((a\succ (b\prec
c))\succ (d\prec z))\succ v\\
&\equiv& 0~mod(S, w_{11.6}).
\end{eqnarray*}

The proof is complete.
 \ \hfill\end{proof}

\begin{definition}\label{dd}
An $L$-word $u$ is called a normal $DD$-word, denoted by $\lceil
u\rceil$, if
\begin{enumerate}
\item[1)] $u = x$,  $x \in X$,
\item[2)] $u = x \prec \lceil v\rceil$, $x\in X$,
\item[3)] $u = x\succ \lceil v\rceil$, $x\in X$,
\item[4)] $u = (x\succ \lceil u_1\rceil)\succ \lceil u_2\rceil$, $x\in X$.
\end{enumerate}
\end{definition}

\noindent \textbf{Remark} \ From Definition \ref{l} and Definition
\ref{dd}, we know that any normal $DD$-word is a normal $L$-word.

The following corollary follows from Theorem \ref{cdl} and Theorem
\ref{GSBl}.

\begin{corollary}\label{cor1}
The set $Irr (S) = \{ u |~u ~\text{is a normal }DD\text{-word }\}$
is a $k$-basis of the free dendriform algebra $DD(X)$.
\end{corollary}

\section{Hilbert series and Gelfand-Kirillov dimension of the free dendriform algebra}

In this section, we give Hilbert series of the free dendriform
algebra $DD(X)$ where $|X|$ is finite. As an application, we prove
that Gelfand-Kirillov dimension of the free dendriform algebra
$DD(X)$ is infinite.

We introduce some basic definitions and concepts that we will use
throughout this section.

\begin{definition}
Let $V=(V,\prec, \succ)$ be a dendriform algebra. Then $V$ is called
a finitely graded algebra if
$$
V = \oplus_{m\geq 1}V_m
$$
as $k$-vector spaces  such that
$$
dim_kV_m < \infty \ \mbox{and }\ \  \delta (V_i, V_j)\subseteq
V_{i+j}\  \ \mbox{ for all } ~ i, j\geq 1, \ \delta\in
\{\prec,\succ\}.
$$
\end{definition}

\begin{definition}
Let $V = \oplus_{m\geq 1}V_m$ be a finitely graded dendriform
algebra and $dim_k(V_m)$, the dimension of the vector space $V_m$.
Then the Hilbert series of $V$ is defined to be
$$
\mathcal {H}(V, t) = \sum_{m=1}^\infty dim_k(V_m)t^m.
$$
\end{definition}

Let $X = \{x_1, x_2,\dots, x_n\}$ and $DD_m$ the subspace of $DD(X)$
 generated by all normal $DD$-words in $DD(X)$ of degree $m$. Then
$$
DD(X) = \oplus_{m\geq 1}DD_m
$$
is a finitely graded dendriform algebra.
\\

By the definition of normal $DD$-words, one has
$$
dim_k(DD_1)=n,~dim_k(DD_2)=2n^2.
$$
Assume that for any $m\geq 1$, $dim_k(DD_ m)= f(m)n^m$. Then
$f(1)=1,\ f(2)=2$. For convenience, let $f(0)=1$.

For any $m>2$, it is clear that $DD_m$ has a $k$-basis
\begin{eqnarray*}
&&\ \ \ \ \{x\prec \lceil u\rceil|x\in X,~|u|>1, \lceil
u\rceil\mbox{ is
a normal }DD\mbox{-word}\}\\
&&\ \ \bigcup \{x\succ \lceil u\rceil|x\in X,~|u|>1\lceil
u\rceil\mbox{ is
a normal }DD\mbox{-word}\}\\
&&\ \ \bigcup \{(x\succ \lceil u_1\rceil)\succ \lceil u_2\rceil|x\in
X,~|u_1|,~|u_2|\geq 1,\ \lceil u_1\rceil, \lceil u_2\rceil\mbox{ are
normal } DD\mbox{-words}\}.
\end{eqnarray*}
It follows that
\begin{eqnarray*}
f(m)&=&2\times f(m-1)+1\times 1\times f(m-2)+1\times f(m-2)\times
1 \\
&&+1\times \sum_{i=2}^{m-3}f(i)f(m-3-i) \\
&=&\sum_{i=0}^{m-1}f(i)f(m-1-i).
\end{eqnarray*}

Therefore, we prove the following lemma.

\begin{lemma}\label{th3}
Let $X$ be a finite set with $|X|=n$. Then the Hilbert series of the
free dendriform algebra $DD(X)$ is
$$
\mathcal {H}(DD(X), t)=\sum_{m\geq1} f(m)n^mt^m,
$$
where $f(m)$ satisfies the recursive relation ($f(0)=1$):
\begin{eqnarray*}
f(m)=\sum_{i=0}^{m-1}f(i)f(m-1-i),\ ~m\geq 1.
\end{eqnarray*}
\end{lemma}

Now, we describe the Hilbert series of $DD(X)$ with another way.
\\

Let $A,~B,~C$ be the subspaces of $DD(X)$ with $k$-bases
\begin{eqnarray*}
&&\{x\prec \lceil u\rceil|~x\in X,\ \lceil u\rceil \mbox{ is a
normal }
DD\mbox{-word}\},\\
&&\{x\succ \lceil u\rceil|~x\in X,\ \lceil u\rceil \mbox{ is a
normal }
DD\mbox{-word}\},\\
&&\{(x\succ \lceil u_1\rceil)\succ \lceil u_2\rceil|~x\in X,\ \lceil
u_1\rceil,\lceil u_2\rceil \mbox{ are normal } DD\mbox{-words}\},
\end{eqnarray*}
respectively. Assume that their Hilbert series are $\mathcal {H}(A,
t),~\mathcal {H}(B, t),~\mathcal {H}(C, t),$ respectively. Clearly,
we have
$$\mathcal {H}(B, t) = \mathcal {H}(A, t).$$

Noting that $A$ has a $k$-basis
$$
\{x_i\prec x_j|~x_i,~x_j\in X\}\bigcup \{x\prec \lceil u\rceil|~|u|>
1,~x\in X,\ \lceil u\rceil \mbox{ is a normal } DD\mbox{-word}\},
$$
we have
\begin{eqnarray}\label{equ1}
\mathcal {H}(A, t) &= &n^2t^2 + nt\times (\mathcal {H}(A, t) +
\mathcal {H}(B, t) + \mathcal {H}(C, t))\nonumber \\
&=& n^2t^2 + nt\times (2\mathcal {H}(A, t)+ \mathcal {H}(C, t)).
\end{eqnarray}

Since $C$ has a $k$-basis
\begin{eqnarray*}
&&\ \{(x_i\succ x_j)\succ x_k|~x_i, x_j, x_k\in X\}\\
&& \bigcup \{(x_i\succ x_j)\succ \lceil u\rceil|x_i, x_j\in X, |u|>
1,\
\lceil u\rceil \mbox{ is a normal } DD\mbox{-word}\}\\
&& \bigcup~ \{(x_i\succ \lceil u\rceil)\succ x_j|x_i, x_j\in X, |u|>
1,\ \lceil u\rceil \mbox{ is a normal } DD\mbox{-word}\}\\
&& \bigcup\{(x\succ \lceil u\rceil)\succ \lceil v\rceil|~|u|, |v|>
1,\ \lceil u\rceil, \lceil v\rceil\mbox{ are normal }
DD\mbox{-words}\},
\end{eqnarray*}
we have
\begin{eqnarray}\label{equ2}
\mathcal {H}(C, t) &= &n^3t^3 + 2n^2t^2\times (\mathcal {H}(A, t) +
\mathcal {H}(B, t) + \mathcal {H}(C, t)) + nt\times (\mathcal {H}(A,
t) \nonumber\\
&&+ \mathcal {H}(B, t) + \mathcal {H}(C, t))^2\nonumber\\
&=&nt\times(nt+(2\mathcal {H}(A, t) + \mathcal {H}(C, t)))^2
\end{eqnarray}

From equations (\ref{equ1}) and (\ref{equ2}), we obtain
$$
\mathcal {H}(A, t) = \frac{1-2nt\pm \sqrt{1-4nt}}{2}.
$$
Since $\mathcal {H}(A, 0) = 0,$ we have
$$
\mathcal {H}(A, t) = \frac{1-2nt- \sqrt{1-4nt}}{2}.
$$
Therefore,
$$
\mathcal {H}(C, t) = \frac{1-(1-2nt)\sqrt{1-4nt}}{2nt}-2+nt.
$$

Thus, we have the following theorem.
\begin{theorem}
 Let $X$ be a finite set with
$|X|=n$. The Hilbert series of the free dendriform algebra $DD(X)$
is
\begin{eqnarray*}
\mathcal {H}(DD(X), t) =\frac{1-2nt-\sqrt{1-4nt}}{2nt}.
\end{eqnarray*}
\end{theorem}

\ \

We now give an exact expression of the function $\mathcal {H}(DD(X),
t)$.

 For $t\leq \frac{1}{4n},$ we have
$$
\sqrt{1-4nt}=(1+(-4nt))^{\frac{1}{2}}=1+\sum_{i=1}^\infty
\frac{\frac{1}{2}(\frac{1}{2}-1)\cdots(\frac{1}{2}-i+1)}{i!}\times(-1)^i4^in^it^i.
$$

From this and Lemma \ref{th3} we get the following theorem.
\begin{theorem}\label{th4}
Let $X$ be a finite set with $|X|=n$. Then the Hilbert series of the
free dendriform algebra $DD(X)$ is
\begin{eqnarray*}
\mathcal {H}(DD(X), t)&=&\sum_{m=1}^\infty
\frac{1\times3\times\cdots(2m-1)\times2^m}{(m+1)!}n^mt^m\\
&=&\sum_{m=1}^\infty \frac{(2m)!\times n^m\times t^m}{(m+1)!m!}.
\end{eqnarray*}
Therefore, $dim_k(DD_m)=\frac{(2m)!\times n^m}{(m+1)!m!},\ m\geq 1$.
\end{theorem}

Now, by using Theorem \ref{th4}, we show that Gelfand-Kirillov
dimension of the free dendriform algebra $DD(X)$ is infinite when
$|X|$ is finite.
\begin{definition}\cite{Kukin}
Let $R$ be a finitely presented algebra over a field $k$ and $x_1,
x_2, \ldots, x_n$ be its generators. Consider $R=\bigcup_{d\in
N}V_{(d)}$, where $V_{(d)}$ is spanned by all the monomials in $x_i$
of length $\leq d$. The quantity
$$
GKR = \overline{\lim_{d\rightarrow \infty}}\frac{log\ dim_k
V_{(d)}}{log\ d}
$$
is called the Gelfand-Kirillov dimension of $R$.
\end{definition}

\begin{theorem}
Let $X$ be a finite set with $|X|=n$.  Then the Gelfand-Kirillov
dimension of free dendriform algebra $DD(X)$ is
$$
GKDD(X)=\infty.
$$
\end{theorem}

\begin{proof}
For a fixed natural $d$, let $DD_{(d)}$ be the subspace spanned by
all the monomials in $x_i$ of length $\leq d$. Then
\begin{eqnarray*}
dim DD_{(d)}&=&\sum_{i=1}^d dim_k(DD_i) \geq dim_k(DD_d).
\end{eqnarray*}

Therefore,
\begin{eqnarray*}
GKDD(X)&\geq & \overline{\lim_{d\rightarrow \infty}}\frac{log\
dim_k(DD_d)}{log\ d} = \overline{\lim_{d\rightarrow
\infty}}\frac{ln\
\frac{(2d)!\times n^d}{(d+1)!d!}}{ln\ d}\\
&= & \overline{\lim_{d\rightarrow \infty}}\frac{d\ ln\ 2n+\sum_{i=1}^dln(2i-1)-\sum_{i=1}^{d+1}ln\ i}{ln\ d}\\
&= & \overline{\lim_{d\rightarrow \infty}}d\
ln2n+\overline{\lim_{d\rightarrow \infty}}\sum_{i=1}^d\frac{ln\frac{2i-1}{i}}{ln\ d}-1\\
&= & \overline{\lim_{d\rightarrow \infty}}(d\ ln2n) +
\overline{\lim_{d\rightarrow \infty}}(\frac{d}{ln\
d}\sum_{i=1}^d\frac{ln\ (2-\frac{1}{i})}{d})-1\\
%&= & \infty + \overline{\lim_{d\rightarrow \infty}}(\frac{d}{ln\
%d}\int_0^1ln\ (2-x)dx)-1\\
%&= & \infty + \overline{\lim_{d\rightarrow
%\infty}}(\frac{1}{\frac{1}{d}}\int_0^1ln\
%(2-x)dx)-1\\
 &= & \infty.
\end{eqnarray*}

 \ \hfill\end{proof}


\begin{thebibliography}{99}
\bibitem{Huang} L.A. Bokut, Yuqun Chen and Jiapeng Huang, Gr\"obner-Shirshov Bases for $L-$algebras, arXiv: 1005.0118v1

\bibitem{Kukin} L.A. Bokut, G.P. Kukin, {\it Algorithmic and Combinatorial Algebra,}
Kluwer Academic Publishers (1994).


\bibitem{bu65}B. Buchberger, An algorithmical criteria for the solvability of
algebraic systems of equations [in German], {\it Aequationes Math.}
{\bf 4}(1970) 374-383.

\bibitem{H}
H. Hironaka, Resolution of singularities of an algebtaic variety
over a field of characteristic zero  I, II, {\it Ann. Math.} {\bf
79} (1964) 109-203, 205-326.

\bibitem{Ph}
P. Leroux, L-algebras, triplicial-algebras, within an equivalence of
categories motivated by grahs, arXiv: 0709.3453v2

\bibitem{Lo01}
J.L. Loday, Dialgebras, in dialgebras and related operads, {\it
Lecture Notes in Mathematics}, Vol. 1763. Berlin: Springer Verl.
 (2001) 7-66.

\bibitem{Lo02} J.L. Loday, Encyclopedia of types of algebras,
http.//www-irma.u-strasbg.fr/~loday/PAPERS/EncyclopALG(root).pdf.

\bibitem{Sh} A.I. Shirshov, Some algorithmic problem for Lie algebras,
{\it  Sibirsk. Mat. Z.} {\bf 3} (1962)  292-296 (in Russian);
English translation in SIGSAM Bull. {\bf 33} (2)  (1999) 3-6.


\bibitem{La} Selected works of A.I. Shirshov, Eds L.A. Bokut, V. Latyshev, I. Shestakov,
E. Zelmanov, Trs M. Bremner, M. Kochetov, Birkh\"auser, Basel,
Boston, Berlin (2009).


\end{thebibliography}
\end{document}